\documentclass[12pt]{amsart}
\usepackage{a4wide}
\usepackage{url}
\newcommand{\abs}[1]{\left\lvert #1 \right\rvert}  

\DeclareMathOperator{\IE}{\mathbf{E}}                     
\newcommand{\FT}{\mathcal{F}}
\newcommand{\IN}{\mathbf{N}}                     
\newcommand{\IR}{\mathbf{R}}                     
\newcommand{\mystirl}[2]{\genfrac{\{}{\}}{0pt}{}{#1}{#2}}
\DeclareMathOperator{\Var}{\mathrm{Var}}                     
\DeclareMathOperator{\Cov}{\mathrm{Cov}}                     
\DeclareMathOperator{\Prob}{\mathbf{P}}                     
\renewcommand{\phi}{\varphi{}}
\DeclareFontFamily{U}{wncy}{}
\DeclareFontShape{U}{wncy}{m}{n}{<->wncyr10}{}
\DeclareSymbolFont{mcy}{U}{wncy}{m}{n}
\DeclareMathSymbol{\Sh}{\mathord}{mcy}{"58}
\DeclareMathSymbol{\sh}{\mathord}{mcy}{"78}

\DeclareMathOperator*{\inint}{\iint\cdots\int}                     
\DeclareMathOperator*{\iints}{\iint}                     

\makeatletter
  \DeclareRobustCommand\em
         {\@nomath\em \ifdim \fontdimen\@ne\font >\z@
                       \upshape \else \slshape \fi}
\makeatother

\newtheoremstyle{mythm}
  {9pt}
  {9pt}
  {\slshape}
  {0pt}
  {\bfseries}
  {.}
  { }
  {\thmname{#1} \thmnumber{ #2}\thmnote{ (#3)}}

\theoremstyle{mythm}

\newtheorem{Theorem}{Theorem}[section]
\newtheorem{Lemma}[Theorem]{Lemma}

\theoremstyle{definition} 
\newtheorem{Definition}[Theorem]{Definition}
\newtheorem{Remark}[Theorem]{Remark}

\numberwithin{equation}{section}

\begin{document}
\title{Cumulants as Iterated Integrals}
\author{Franz Lehner}
\address{Institut f\"ur Mathematische Strukturtheorie,
TU Graz,\\
Steyrergasse 30,\\
8010 Graz, Austria}
\email{lehner@finanz.math.tu-graz.ac.at}
\begin{abstract}
  A formula expressing cumulants in terms of iterated integrals
  of the distribution function is derived.
  It generalizes results of Jones and Balakrishnan
  who computed expressions for cumulants up to order 4.
\end{abstract}
\date{May 20, 2009}
\maketitle{}

\section{Introduction}
The expectation of a random variable can be computed in many ways.
One method involving only the distribution function
is obtained by careful partial integration and looks as follows:
$$
  \IE X = \int_{0}^\infty (1-F(t))\,dt - \int_{-\infty}^0 F(t)\,dt
;
$$
a similar formula holds for the  variance,
which can be written as the following double integral:
\begin{equation}
  \label{eq:CumulantShuffles:variance}
\Var(X)=2\iints_{-\infty<t_1<t_2<\infty} F(t_1)(1-F(t_2))\,dt_1\,dt_2
.
\end{equation}
Analogues of these formulas expressing
the third an fourth cumulants (skewness and kurtosis)
in terms of iterated integrals of the distribution function
were computed some time ago by
Jones and Balakrishnan~\cite{JonesBalakrishnan:2002:moments}.
The proof relied on \emph{ad hoc} partial integration,
see also \cite{BassanDenuitScarsini:1999:variability},
where similar formulas for mean differences are considered.

The aim of the present note is a generalization of
these expressions to cumulants of arbitrary order,
resulting in a formula resembling the
well-known M\"obius inversion formula,
which expresses cumulants in terms of moments.

The paper is organized as follows.
After a short introduction to cumulants in 
Section~\ref{sec:CumulantShuffles:moments}
and a review of partitions and shuffles in Section~\ref{sec:partitions}
we give two proofs of the main result.
The first one using an elementary identity for the Volterra integral operator
and Chen's shuffle formula for multiple integrals is contained in
Section~\ref{sec:iterated}. In the
concluding Section~\ref{sec:hoeffding} we give
another proof based on a formula
for multivariate cumulants due to Block and Fang.

\section{Moments and cumulants}
\label{sec:CumulantShuffles:moments}
Let $X$ be a random variable with distribution function $F(x)=P(X\leq x)=\int_{-\infty}^x dF(t)$.
Its \emph{moments} are the numbers
$$
m_n = \IE X^n = \int_{-\infty}^\infty t^n \,dF(t)
$$

Under some assumptions the sequence of moments contains
the complete information about the distribution
of $X$. It can be collected in the exponential \emph{moment generating
  function} 
(formal Fourier-Laplace transform)
$$
\FT_X(z) = \IE e^{zX}
= \sum_{n=0}^\infty \frac{m_n}{n!}\,z^n
= 1 + \frac{m_1}{1!}\,z+\frac{m_2}{2!}\,z^2+\cdots
$$
and the Taylor coefficients of the formal logarithm of the m.g.f.
$$
\log \FT_X(z) = \sum_{n=1}^\infty \frac{\kappa_n}{n!}\,z^n
$$
are called the \emph{cumulants}.
The first two are expectation and variance
\begin{align*}
\kappa_1&=m_1=\IE X\\  
\kappa_2&=m_2-m_1^2=\Var X
\end{align*}
and after rescaling the following two are 
the skewness $\kappa_3/\kappa_2^{3/2}$ and the kurtosis
$\kappa_4/\kappa_2^2$.

The cumulants carry the same information as the moments
but for many purposes in a better digestible form,
e.g., after a translation the moments behave like
$$
m_n(X+\tau) = \sum_{k=0}^n \binom{n}{k} \tau^{n-k} m_k(X)
$$
while the cumulants are
\begin{equation}
  \label{eq:CumulantShuffles:TranslatedCumulant}
  \kappa_n(X+\tau) =
  \begin{cases}
    \tau+\kappa_1(X) & n=1\\
    \kappa_n(X)      & n\geq 2
  \end{cases}
\end{equation}
For this reason the cumulants are sometimes called the \emph{semi-invariants} of $X$.
The most important property of the cumulants is the identity
$$
\kappa_n(X+Y)=\kappa_n(X)+\kappa_n(Y)
$$
if $X$ and $Y$ are independent.

\section{Partitions and Shuffles}
\label{sec:partitions}
There is also a combinatorial formula expressing the cumulants
as polynomials of the moments.
A \emph{set partition} of order $n$ is set
$$
\pi=\{B_1,B_2,\dots,B_p\}
$$
of of disjoint subsets $B_i\subseteq\{1,2,\dots,n\}$, called \emph{blocks},
whose union is $\{1,2,\dots,n\}$.
Denote $\Pi_n$ the set of all $n$-set partitions. It is a lattice under the refinement order
$$
\pi\leq\sigma\iff \text{every block of $\pi$ is contained in a block of $\sigma$},
$$
with minimal element $\hat{0}_n=\{\{1\},\{2\},\dots,\{n\}\}$
and maximal element $\hat{1}_n=\{\{1,2,\dots,n\}\}$.
Each set partition $\pi\in\Pi_n$ determines a number partition,
called its \emph{type},
$\lambda(\pi)=1^{k_1}2^{k_2}\cdots n^{k_n}\vdash n$
where $k_j$ is the number of blocks $B\in\pi$ of size $\abs{B}=j$.
For the combinatorial identities below we will employ 
the following conventions.
For a partition $\lambda = 1^{k_1}2^{k_2}\cdots n^{k_n}\vdash n$ we abbreviate
$\lambda!=1!^{k_1}2!^{k_2}\dotsm n!^{k_n}$ and
for a sequence $(a_n)_{n\in\IN}$ of numbers we denote 
$$
a_\lambda = \prod a_j^{k_j}
;
$$
similarly for a set partition $\pi$ we let
$$
a_\pi = a_{\lambda(\pi)} = \prod_{B\in\pi} a_{\abs{B}}
.
$$
Given a partition $\lambda\vdash n$,
the number of set partitions $\pi\in\Pi_n$ with $\lambda(\pi)=\lambda$
is equal to the \emph{Faa di Bruno coefficient}
\begin{equation}
  \label{eq:CumulantShuffles:FaaDiBruno}
\mystirl{n}{\lambda}=
\#\{\pi:\pi \sim \lambda\} 
= \frac{n!}%
   {1!^{k_1} 2!^{k_2} \cdots n!^{k_n} k_1!k_2!\cdots k_n!}
\end{equation}

The well known moment-cumulant formula says
\begin{equation}
  \label{eq:CumulantShuffles:Schuetzenberger}
\kappa_n = \sum_{\pi\in\Pi_n} m_\pi \,\mu(\pi,\hat{1}_n)
\end{equation}
where $\mu$ is the M\"obius function on $\Pi_n$. Its values only depend on
$\lambda(\pi)=1^{k_1}2^{k_2}\cdots n^{k_n}$,
namely 
$$
\mu(\pi,\hat{1}_n)=\mu_{\lambda(\pi)}=\prod_{j=1}^n ((-1)^j(j-1)!)^{k_j}
$$
Using the Faa Di Bruno coefficients~\eqref{eq:CumulantShuffles:FaaDiBruno}
the moment-cumulant formula can be condensed to
\begin{equation}
  \label{eq:CumulantShuffles:SchuetzenbergerFaaDiBruno}
\kappa_n = \sum_{\lambda\vdash n} 
\mystirl{n}{\lambda}
\,
m_\lambda \,\mu_\lambda
\end{equation}

\begin{Definition}
  \label{def:CumulantShuffles:shuffle}
Let $a=(a_1,a_2,\dots, a_m)$ and  $b=(b_1,b_2,\dots, b_m)$ be two finite
sequences.
A \emph{shuffle} of $a$ and $b$ is a pair of order preserving injective maps
$\phi:a\to \{1,\dots,m+n\}$ and $\psi:b\to \{1,\dots,m+n\}$ 
with disjoint images. 
When the two sequences are thought of as two decks of cards,
this corresponds to putting the two decks together in such a way
that the relative order in the individual decks is preserved,
the card $a_i$ (resp.\ $b_j$) being put in position $\phi(a_i)$ (resp.\
$\psi(b_j)$). The result of the shuffle is the sequence, where each $i$
is replaced by the symbol $\phi^{-1}(i)$ (or $\psi^{-1}(i)$).
Denote $\Sh(m,n)$ the set of shuffles of the sequences $(1,2,\dots,m)$ and
$(m+1,m+2,\dots,m+n)$.
Shuffles of multiple sequences are defined accordingly
and for a partition $\lambda\vdash n$ denote $\Sh(\lambda)$ the set of shuffles
of disjoint sequences with cardinalities given by $\lambda$.
\end{Definition}

Each shuffle is uniquely determined by the subsets the individual sequences are mapped
to and thus the number of shuffles is equal to the number of ways of picking
these subsets, i.e., the multinomial coefficient
$$
\#\Sh(\lambda) = \binom{n}{\lambda}
$$
On the other hand
each shuffle $\tau\in\Sh(\lambda)$ determines a partition $\pi$ of $\{1,2,\dots,n\}$
of type $\lambda$. However different shuffles may determine the same
partition $\pi$, if $\lambda$ contains entries of the same size.
Therefore we have the identity
$$
\#\Sh(\lambda)=k_1!k_2!\cdots k_n! \mystirl{n}{\lambda}
$$
which will lead to an interesting cancellation later.

\section{Iterated Integrals}
\label{sec:iterated}

\begin{Definition}
  The \emph{Volterra operator} is the integral operator
  $$
  Vf(x) = \int_{-\infty}^x f(t)\,dt
  $$
  defined for suitable integrable functions $f:\IR\to\IR$;
  its powers are defined recursively by
  $$
  V^n f(x) = \int_{-\infty}^x (V^{n-1}f)(t)\,dt
  $$
\end{Definition}
It was first observed by Chen~\cite{Chen:1957:integration} and Ree
\cite{Ree:1958:lie}
that the recursively defined iterated integrals
$$
\alpha_{i_1,i_2,\dots,i_n}(t) = \int_a^t \alpha_{i_1,i_2,\dots,i_{n-1}}(u)\,d\alpha_{i_n}(u)
$$
which can be written as
$$
\inint_{a<t_1<t_2<\cdots t_n} d\alpha_{i_1}(t_1)\,d\alpha_{i_2}(t_2)\cdots d\alpha_{i_n}(t_n)
$$
satisfy the shuffle relations
$$
\alpha_{i_1,i_2,\dots,i_m}(t) \, \alpha_{j_1,j_2,\dots,j_n}(t) 
= \sum_{\sigma\in\Sh(\{i_1,\dots,i_m\},\{j_1,\dots,j_n\})} \alpha_\sigma(t)
$$
For the Volterra operator this means that for example
$$
V^mf(x) \, V^nf(x) = \sum_{\tau\in \Sh(m,n)} \inint_{t_1<\cdots<t_{m+n}}
 \, f_\tau(t_1,\dots,t_{m+n})\,dt_1\,dt_2\cdots dt_{m+n}
;
$$
where $f_\tau(t_1,\dots,t_{m+n})=f(t_{\phi(1)})\,f(t_{\psi(1)})$
with $\phi$ and $\psi$ as in Definition~\ref{def:CumulantShuffles:shuffle}.

Note that if $f(t)$ is a probability density function, then
the corresponding distribution function is given by
$$
F(x) = Vf(x)
$$
We denote $F^{[n]}(x)=VF^{[n-1]}(x) = V^{n+1}f(x)$
where $F^{[0]}(x)=F(x)=\int_{-\infty}^x dF(t)$.
This notation slightly differs from \cite{BassanDenuitScarsini:1999:variability}.
Then one can easily show by induction that these integrals are truncated moments.
\begin{Lemma}
  $$
  F^{[n]}(\tau) = \frac{1}{n!}
             \int_{-\infty}^\tau 
              (\tau-t)^n\,dF(t)
  = \frac{1}{n!}\IE (\tau-X)_+^n
  $$
\end{Lemma}
Before proceeding further,
assume for the moment that the support of $dF$ is bounded and that $\tau$ 
is an upper bound.
Then $Y=X-\tau$ has moments
$$
y_n = \IE (X-\tau)^n = (-1)^n n! F^{[n]}(\tau)
=
\sum_{k=0}^n \binom{n}{k}
(-\tau)^{n-k} m_k
$$
and because of  \eqref{eq:CumulantShuffles:TranslatedCumulant}
the cumulants are
$$
\kappa_n(X)=
\begin{cases}
  \kappa_1(Y)+\tau & n=1\\
  \kappa_n(Y) & n\geq 2
\end{cases}
.
$$
Thus we can express the cumulants of order $n\geq2$ 
by a moment-cumulant type formula in terms of $y_n=(-1)^nn! V^nF$:
\begin{equation}
\label{eq:CumulantShuffles:kappan=sumypi}
  \kappa_n(X)
  = \sum_{\pi\in\Pi_n} y_\pi\, \mu(\pi,\hat{1}_n)
\end{equation}
After some shuffling (in the literal sense!)
this simplifies to the following formula.
\begin{Theorem}
\label{thm:CumulantShuffles:univariate}
\begin{equation}
\label{eq:CumulantShuffles:univariate}
  \kappa_n
  = (-1)^nn! \inint_{t_1<t_2<\dots< t_n} 
  \sum_{\pi\in\Pi_n} F_\pi(t_1,t_2,\dots,t_n)\,
  \mu(\pi,\hat{1}_n)\, dt_1\,dt_2\cdots dt_n
\end{equation}
where $F_\pi(t_1,t_2,\dots,t_n)=\prod_{B\in\pi} F(t_{\alpha(B)})$ 
where $\alpha(B)$ denotes the first (i.e., smallest) element of a block $B$.
\end{Theorem}
Note that the
higher order cumulants do not depend on $\tau$ and thus
the formula also holds if the support of $X$ is unbounded.

\begin{Remark}

  It may be hoped that this formula provides some insight to Rota's
  \emph{problem of the cumulants} \cite{Rota:2001:twelve},
  namely to find a collection of inequalities which are necessary
  and sufficient
  for a number sequence to be the cumulant sequence of some probability
  distribution.
  One advantage of formula \eqref{eq:CumulantShuffles:univariate}
  is the fact that all terms appearing in the sum 
  are either nonnegative or nonpositive, regardless
  which particular probability distribution is considered,
  in contrast to the moments in Sch\"utzenberger's formula
  \eqref{eq:CumulantShuffles:Schuetzenberger}.
\end{Remark}

\begin{proof}
  Assume first that the support of $X$ is bounded.
  Starting from \eqref{eq:CumulantShuffles:kappan=sumypi},
  or rather  \eqref{eq:CumulantShuffles:SchuetzenbergerFaaDiBruno},
  we have
  \begin{align*}
    \kappa_n
    &= \sum_{\lambda\vdash n} \mystirl{n}{\lambda}\,y_\lambda \,\mu_\lambda \\
    &= (-1)^n \sum_{\lambda\vdash n} \mystirl{n}{\lambda} \lambda!
    \,F^{[\lambda]}(\tau) \,\mu_\lambda \\    
\intertext{and by Chen's lemma this is equal to} 
    &= (-1)^n \sum_{\lambda\vdash n} \mystirl{n}{\lambda} \lambda!
    \sum_{\sigma\in\Sh(\lambda)}
    \inint_{-\infty<t_1<t_2<\cdots<t_n<\tau}F_\sigma(t_1,\dots,t_n)\,dt_1\,dt_2\cdots dt_n
     \,\mu_\lambda . \\    
\intertext{Since $F_\sigma$ only depends on the partition determined by
  $\sigma$, we can collect equal terms to get}
    &= (-1)^n \sum_{\lambda\vdash n} \mystirl{n}{\lambda} \lambda!k_1!k_2!\dots k_n!
    \sum_{\pi\sim\lambda} \inint_{-\infty<t_1<t_2<\cdots<t_n<\tau}F_\pi(t_1,\dots,t_n)\,dt_1\,dt_2\cdots dt_n
     \,\mu_\lambda \\
    &= (-1)^n n!\sum_{\pi\in\Pi_n} \inint_{-\infty<t_1<t_2<\cdots<t_n<\tau}F_\pi(t_1,\dots,t_n)\,dt_1\,dt_2\cdots dt_n
     \,\mu(\pi,\hat{1}_n)
    .
  \end{align*}
  Since the final formula does not depend on the chosen integration bound
  $\tau$, we may let it go to infinity and the formula holds for arbitrary
  distribution functions.
\end{proof}
\begin{Remark}
  Using the recursive structure of the partition lattice, we can partially factorize
  the integrand of \eqref{eq:CumulantShuffles:univariate}.
  Each partition $\pi\in\Pi_n$ can be constructed from a unique partition $\pi'\in\Pi_{n-1}$
  by either adjoining $\{n\}$ as a separate block or by joining $n$ to one of the blocks
  of $\pi'$.
  Assume that $\pi'$ has $k$ blocks.
  In the first case the number of blocks is increased  to $k+1$
  and the M\"obius function, which depends on the number of blocks,
  becomes $\mu(\pi,\hat{1}_n)=-k\mu(\pi',\hat{1}_{n-1})$.
  The integrand changes simply to $F_\pi(t)=F_{\pi'}(t)\,F(t_n)$.

  In the second case the number of blocks remains the same and also the
  M\"obius function and integrand stays the same. However there are $k$ possible ways
  to join $n$ to a block of $\pi'$.
  Thus we can write
  \begin{align*}
    \kappa_n(X)
    &= (-1)^n\,n!
       \inint_{t_1<t_2<\dots< t_n} 
       \sum_{k=1}^{n-1}
       \sum_{\pi'\in\Pi_{n-1,k}} (-k F_{\pi'}(t_1,t_2,\dots,t_{n-1})\,F(t_n) 
\\& \hskip15em\hfill
+ k F_{\pi'}(t_1,t_2,\dots,t_{n-1})
\,
    \mu(\pi',\hat{1}_{n-1})\, dt_1\,dt_2\cdots dt_n \\
    &= (-1)^n\,n!
       \inint_{t_1<t_2<\dots< t_n} 
       \sum_{\pi\in\Pi_{n-1}} \abs{\pi} \, F_{\pi}(t_1,t_2,\dots,t_{n-1})\,(1-F(t_n))\,
    \mu(\pi,\hat{1}_{n-1})\, dt_1\,dt_2\cdots dt_n \\
  \end{align*}
  
  The number of integration variables can still be reduced by two
  as discussed in \cite{Jones:2004:expressions}.
  For this purpose we introduce the so-called
  \emph{mean redidual life functions}
  of Barlow and Proschan~\cite{BarlowProschan:1965:reliability},
  namely
  \begin{align*}
    R(y) &= \IE(X|X>y) - y = \frac{\int_y^\infty (1-F(t))\,dt}{1-F(y)}\\
    P(y) &= y - \IE(X|X<y) = \frac{\int_{-\infty}^y F(t)\,dt}{F(y)}
  \end{align*}
  and obtain
  $$
  \kappa_n(X) =
    (-1)^n\,n!
    \inint_{t_2<\dots< t_{n-1}} 
    \sum_{\pi\in\Pi_{n-1}} \abs{\pi} \, P(t_2)\,F_{\pi}(t_2,t_2,\dots,t_{n-1})\,(1-F(t_{n-1}))R(t_{n-1})\,
       \mu(\pi,\hat{1}_{n-1})\, dt_2\,dt_3\cdots dt_{n-1} \\
  $$
  because $F_\pi(t_1,t_2,\dots,t_{n-1})$ always contains $F(t_1)$ as a factor.

\end{Remark}

\section{H\"offding's formula and multivariate cumulants}
\label{sec:hoeffding}
In this section we give another proof of Theorem~\ref{thm:CumulantShuffles:univariate}
using an extension of H\"offding's formula due to Block and 
Fang~\cite{BlockFang:1988:multivariate}. 
Multivariate cumulants are defined as coefficients of multivariate Fourier transforms,
namely
$$
\kappa_n(X_1,X_2,\dots,X_n)=
  =
  \left.
    \frac{\partial}{\partial z_1} \dotsm \frac{\partial}{\partial z_n}
    \log \IE e^{z_1 X_1 + \dots + z_n X_n}
  \right|_{z_1=\cdots=z_n=0}
$$
The univariate cumulants in Section~\ref{sec:CumulantShuffles:moments} correspond to the 
case $X_1=X_2=\cdots=X_n=X$.
As an example,
the second cumulant is the covariance
$$
\kappa_2(X_1,X_2) = \Cov(X_1,X_2) = \IE X_1X_2-\IE X_1 \IE X_2
.
$$
In general the multivariate analogue of \eqref{eq:CumulantShuffles:Schuetzenberger}
is
\begin{equation}
  \label{eq:CumulantShuffles:SchuetzenbergerMultivariate}
  \kappa_n(X_1,X_2,\dots,X_n) = \sum_{\pi\in\Pi_n} m_\pi(X_1,X_2,\dots,X_n) \,\mu(\pi,\hat{1}_n)
\end{equation}
where for a partition $\pi\in\Pi_n$ we denote
$$
m_\pi(X_1,X_2,\dots,X_n) = \prod_{B\in\pi} \IE\prod_{i\in B} X_i
.
$$

H\"offding's formula \cite{Hoeffding:1940:masstabinvariante,Lehmann:1966:concepts}
gives an alternative expression for the covariance in terms
of distribution functions:
$$
\kappa_2(X_1,X_2) = \iint (F(t_1,t_2)-F_1(t_1)\,F_2(t_2))\,dt_1\,dt_2
$$
where $F(t_1,t_2)=\Prob(X_1\leq t_1 \wedge X_2\leq t_2)$ is the joint distribution function
and $F_1$ and $F_2$ are the marginal distribution functions of $X_1$ and $X_2$.
From this it is easy to derive \eqref{eq:CumulantShuffles:variance}.

We shall use an extension of H\"offding's formula
to cumulants of all orders, due to Block and 
Fang~\cite{BlockFang:1988:multivariate}
from which Theorem~\ref{thm:CumulantShuffles:univariate} 
follows as a corollary.
For the reader's convenience, we provide a short proof of Block and Fang's
result here. 
For that purpose we introduce the following notations.

Let $X_1,X_2,\dots,X_n$ be random variables and denote by 
$F(t_1,t_2,\dots,t_n)=\Prob(X_1\leq t_1,X_2\leq t_2,\dots,X_n\leq t_n)$
their joint distribution function. More generally,
 for a subset $I\subseteq \{1,2,\dots,n\}$
denote 
$$
F_I((t_i)_{i\in I}) = \Prob(X_i\leq t_i:i\in I)
.
$$
Define iterated integrals by the recursion
\begin{equation}
  \label{eq:CumulantShuffles:MultivariateIteratedRecursion}
\begin{aligned}
  F^{[0,0,\dots,0]}(x_1,x_2,\dots,x_n)&=F(x_1,x_2,\dots,x_n) \\
  F^{[k_1,k_2,\dots,k_i+1,\dots,k_n]}(x_1,x_2,\dots,x_n)&=
  \int_{-\infty}^{x_i} F^{[k_1,k_2,\dots,k_i,\dots,k_n]}(x_1,x_2,\dots,t_i,\dots x_n)\,dt_i
\end{aligned}
\end{equation}
Then one can show by induction that
\begin{equation}
  \label{eq:CumulantShuffles:Fk=Ex-Xp}
F^{[k_1,k_2,\dots,k_n]}(x_1,x_2,\dots,x_n)=
\IE \frac{(x_1-X_1)_+^{k_1}}{k_1!}
    \frac{(x_2-X_2)_+^{k_2}}{k_2!}
    \dotsm
    \frac{(x_n-X_n)_+^{k_n}}{k_n!}
\end{equation}

\begin{Theorem}[{\cite{BlockFang:1988:multivariate}}]
  For any $n\geq2$ the $n$-th multivariate cumulant is given by
  $$
  \kappa_n(X_1,X_2,\dots,X_n) = (-1)^n \inint_{-\infty}^\infty \sum_{\pi\in\Pi_n} F_\pi(t_1,t_2,\dots,t_n)\,\mu(\pi,\hat{1}_n) \,dt_1\,dt_2\dots dt_n
  $$
  where for any partition $\pi\in\Pi_n$ we denote
  $$
  F_\pi(t_1,t_2,\dots,t_n) = \prod_{B\in \pi} F_B(t_i:i\in B)
  $$
\end{Theorem}

\begin{proof}
  Assume first that the random variables $X_i$ are bounded from above
  and choose upper bounds $x_i$.
  Then the subscripts in \eqref{eq:CumulantShuffles:Fk=Ex-Xp} disappear and
  we have
  $$
  F^{[1,1,\dots,1]}(x_1,x_2,\dots,x_n)=
  \IE (x_1-X_1)(x_2-X_2)\cdots(x_n-X_n)
  ;
  $$
  because of translation semi-invariance we may use the modified moment-cumulant formula
  \begin{align*}
  \kappa_n(X_1,X_2,\dots,X_n)
  &= (-1)^n   \kappa_n(x_1-X_1,x_2-X_2,\dots,x_n-X_n)\\
  &= (-1)^n \sum_{\pi\in\Pi_n} F_\pi^{[1,1,\dots,1]}(x_1,x_2,\dots,x_n)
  \end{align*}
  where $F_\pi^{[1,1,\dots,1]}(x_1,x_2,\dots,x_n) = \prod_{B\in\pi} F_B^{[1,1,\dots,1]}(x_i:i\in B)$.
  Writing this out in terms of the recursion
  \eqref{eq:CumulantShuffles:MultivariateIteratedRecursion}
  and noting that the result does not depend on the choice of the bounds $x_i$
  yields the claimed formula.
\end{proof}

\begin{proof}[Second Proof of Theorem~\ref{thm:CumulantShuffles:univariate}]
  If $X_1=X_2=\dotsm =X_n$
  then the integrand
  $$
  \sum_{\pi\in\Pi_n} F_\pi(t_1,t_2,\dots,t_n)\,\mu(\pi,\hat{1}_n)
  $$
  is symmetric in $t_1,t_2,\dots,t_n$, we may ``shuffle'' the integration 
  variables and get
  $$
  \kappa_n(X_1,X_2,\dots,X_n) = (-1)^n n!
  \inint_{-\infty<t_1<t_2<\dots<t_n<\infty}
  \sum_{\pi\in\Pi_n} F_\pi(t_1,t_2,\dots,t_n)\,\mu(\pi,\hat{1}_n) \,dt_1\,dt_2\dots dt_n
  $$
  and now observing that the joint distribution function of any subset
  satisfies
  $$
  F_I(t_i:i\in I)=F(\min(t_i:i\in I))
  $$
  we arrive at the claimed formula.
\end{proof}
\providecommand{\bysame}{\leavevmode\hbox to3em{\hrulefill}\thinspace}
\providecommand{\MR}{\relax\ifhmode\unskip\space\fi MR }
\providecommand{\MRhref}[2]{%
  \href{http://www.ams.org/mathscinet-getitem?mr=#1}{#2}
}
\providecommand{\href}[2]{#2}

\end{document}